\documentclass[journal]{IEEEtran}
\usepackage{epsfig,amssymb,latexsym,color,amsmath,pifont,colordvi,multicol}
\newtheorem{theorem}{Theorem}

\newtheorem{assumption}[theorem]{Assumption}

\newtheorem{example}[theorem]{Example}

\newtheorem{definition}[theorem]{Definition}

\newtheorem{lemma}[theorem]{Lemma}

\usepackage{cite}

\newtheorem{proof}[theorem]{Proof}

\ifCLASSINFOpdf
\else
\fi
\begin{document}
%
\title{Stochastic Stability Analysis of Discrete Time System Using Lyapunov Measure}
%
%
%


\author{Umesh Vaidya,~\IEEEmembership{Senior Member,~IEEE,}
\thanks{Financial support from the National Science Foundation grant ECCS-1150405 and CNS-1329915 is gratefully acknowledged. U. Vaidya is with the Department
of Electrical and Computer Engineering, Iowa State University, Ames,
IA, 50011 USA e-mail: ugvaidya@iastate.edu.}
}

\maketitle

\begin{abstract}
In this paper, we study the stability problem of a stochastic, nonlinear, discrete-time system. We introduce a linear transfer operator-based Lyapunov measure as a new tool for stability verification of stochastic systems. Weaker set-theoretic notion of almost everywhere stochastic stability is introduced and verified, using Lyapunov measure-based stochastic stability theorems. Furthermore, connection between Lyapunov functions, a popular tool for stochastic stability verification, and Lyapunov measures is established. Using the duality property between the linear transfer Perron-Frobenius and Koopman operators, we show the Lyapunov measure and Lyapunov function used for the verification of stochastic stability are dual to each other. Set-oriented numerical methods are proposed for the finite dimensional approximation of the Perron-Frobenius operator;  hence, Lyapunov measure is proposed. Stability results in finite dimensional approximation space are also presented. Finite dimensional approximation is shown to introduce further weaker notion of stability referred to as coarse stochastic stability. The results in this paper extend our earlier work on the use of Lyapunov measures for almost everywhere stability verification of deterministic dynamical systems (``Lyapunov Measure for Almost Everywhere Stability", {\it IEEE Trans. on Automatic Control}, Vol. 53, No. 1, Feb. 2008).
\end{abstract}


%
\IEEEpeerreviewmaketitle

\section{Introduction}
Stability analysis and control of stochastic systems are a problem of theoretical and applied interests.
For stochastic systems, there are various notions of stabilities. Among the most popular notions of stabilities are almost sure and moment stability  \cite{Hasminskii_book,arnold_book_rds}. Almost sure notion of stability implies stability of sample-path trajectories of stochastic systems and is a weaker notion of stability compared to  the moment stability definition. Moment stability definition deals with the steady state probability density function of a stochastic system, in particular, the  integrability of it with respect to various powers of state. There is extensive literature on stochastic stability and stabilization in control and dynamical system literature. Some of the classic results on stability analysis and control of stochastic dynamical system are found in \cite{Hasminskii_book, arnold_book_rds,Kushner_book}. Mao \cite{mao_book} presents a systematic summary of results on  various stochastic stability definitions and Lyapunov function-based verification techniques for a stochastic system. The topic of stochastic stability of switched system and the Markov jump system has also attracted lots of attention with applications in network controlled dynamical systems \cite{network_basar_tamar, networksystems_foundation_sastry,scl04,amit_erasure_observation_journal,amit_ltv_journal,lure_stab_journal,sai_cdc_stochastic_linear}. The results in these papers address not only the stochastic stability problem, but also robust control synthesis and fundamental limitations issues that arise in stabilization and estimation of dynamical systems in the presence of stochasticity in a feedback loop. More generally, stochastic stability and stabilization problems for a nonlinear system with a multiplicative noise process are addressed in \cite{stochastic_kristic2,stochastic_kristic3,Vaidya_erasure_SCL}.

  As  in a deterministic system, most of the existing methods for stochastic stability verification are based directly or indirectly on the Lyapunov function. The Lyapunov function is used for both almost sure and moment stability verification. The Lyapunov function and Lyapunov function-based methods are also applied for stochastic stabilization \cite{stochastic_florchingeer1,stochastic_florchingeer2,stochastic_florchingeer3, stochastic_kristic1,stochastic_kristic4,stochastic_kristic5}. While there is extensive literature on the numerical procedure for the construction of a Lyapunov function for stability verification and stabilization of a deterministic dynamical system \cite{SOS_book,Prajna04,Parrilothesis,Junge_Osinga,Junge_scl_05,cell-cell1,hernandez,mey99a,Lasserre_ocp,book_mceneaney}, the literature on numerical methods for stochastic stability verification is very scant \cite{Grune_04,cell-cell2}. The contributions of this paper are two-fold. We provide a novel operator theoretical framework for stability verification of a stochastic dynamical system. We also show the proposed operator theoretic framework is amicable to computations thereby providing systematic numerical procedures for stochastic stability verification.

  The operator theoretic framework that we introduce in this paper is used to verify a weaker set-theoretic notion of almost everywhere stability of a stochastic system. The notion of almost everywhere stability was introduced for the first time in the work of Rantzer for continuous time deterministic dynamical systems \cite{Rantzer01}. This notion of stability was later extended to continuous time stochastic systems in \cite{almosteverywhere_stochastic}. The notion of a.e. stability essentially implies  the set of points in the state space starting from which system trajectories are not  attracted to the attractor set is a measure zero set.
   The almost everywhere notion of stability was also used for the design of stabilizing feedback controller and for solving verification problems for a nonlinear system \cite{Prajna04}. The input-output version of almost everywhere stability was developed in \cite{angeli_aeinputputput}. In \cite{Vaidya_TAC}, linear transfer operator-based framework was introduced to verify this weaker set-theoretic notion of a.e. stability for discrete-time dynamical system. Transfer Perron-Frobenius operator-based Lyapunov measure was introduced as a new tool to verify a.e stability of a nonlinear system. Duality between the Lyapunov function and Lyapunov measure was established using the duality in the transfer Perron-Frobenius and Koopman operators. Application of the Lyapunov measure for the design of stabilizing and optimal feedback controller was proposed in \cite{Vaidya_CLM_journal,arvind_ocp_journal_IEEE}. Operator theoretical framework involving spectral analysis of Koopman operator was proposed in \cite{mezic_koopman_stability} for stability analysis of deterministic nonlinear systems. The results in this paper can be viewed as a natural extension of the results from \cite{Vaidya_TAC} towards a.e. stability verification of discrete time stochastic dynamical system.

Associated with the stochastic dynamical system are two linear transfer operators, the Perron-Frobenius (P-F) and Koopman operators. These operators are used to study the evolution of ensembles of points in the form of measures or density supported on state space.  These operators are dual to each other and are used in the dynamical system literature for the analysis of deterministic and stochastic dynamical systems \cite{Lasota,Dellnitztransport,Dellnitiz_almostinvariant,mezic_koopmanism,Mezic_comparison,mezic_chaos,froyland_ocean,froyland_extracting}.  One of the main contributions of this paper is the introduction of a linear transfer P-F operator based Lyapunov measure for a.e. stochastic stability verification of stochastic dynamical system. We introduce a.e. almost sure notion of stability and provide Lyapunov measure-based stability theorems to verify this stability. The Lyapunov function for stochastic stability verification is shown  intimately connected with Koopman operator formalism. In particular, analytical formulas for the computation of  Lyapunov measure and the Lyapunov function are obtained in terms of resolvent of the Perron-Frobenius and Koopman operator respectively.  By exploiting the duality relationship between Koopman and Perron-Frobenius operators, we show the Lyapunov function and Lyapunov measure are dual to each other in a stochastic setting as well. The results presented in this paper are extended version of results appeared in \cite{vaidya_stochastic_lyap,vaidya_stochastic_lyap_comp}.

While results exist for the application of operator theoretical methods for the stability verification of a stochastic system in more general Markov chain settings \cite{sean_meyn}, the main motivation of this work is to provide computational methods for the construction of stability certificate in the form of a Lyapunov measure. Towards this goal, set-oriented methods  are used for the finite dimensional approximation of the linear transfer P-F operator and for computation of the Lyapunov measure \cite{Dellnitz_Junge,Dellnitz00}. The finite dimension approximation introduces a further weaker notion of a.e. stability by allowing stable dynamics in the complement of the attractor set, but their domain of attraction is smaller than the size of discretization cells used in the finite dimensional approximation. This notion of stability is referred to as coarse stochastic stability. The finite dimensional approximation of the P-F operator arises as a Markov matrix and provides for various alternate formulas for the computation of the Lyapunov measure.

Organization of this paper is as follows. In section \ref{section_prelim}, we discuss the preliminaries of the transfer operators and introduce various set-theoretic stochastic stability definitions. The main results of this paper on Lyapunov measure-based, stochastic stability theorems are proven in section \ref{section_main}. The formula for  obtaining the Lyapunov function for  a stochastic system in terms of the resolvent of the Koopman operator is presented in section \ref{section_koopman}. The connection between the Lyapunov measure and the Lyapunov function is discussed in section \ref{section_relation}. The set-oriented numerical method for the finite dimensional approximation of P-F operator and stability results using the finite dimensional approximation are discussed in \ref{section_discrete} and \ref{section_coarse} respectively. Simulation examples and discussed in Section \ref{section_examples} followed by conclusions in section \ref{section_conclusion}.

\section{Preliminaries}\label{section_prelim}
The set-up and preliminaries for this section are adopted from \cite{Lasota}. Consider the discrete-time stochastic dynamical system,
\begin{eqnarray}
x_{n+1}=T(x_n,\xi_n),\label{rds}
\end{eqnarray}
where $x_n\in X\subset \mathbb{R}^d$ is a  compact set. We denote ${\cal B}(X)$ as the $\sigma$-algebra of Borel subsets of $X$. The random vectors, $\xi_0,\xi_1,\ldots$, are assumed  independent identically distributed (i.i.d) and  takes values in $W$ with the following probability distribution,
\[{\rm Prob}(\xi_n\in B)=v(B),\;\;\forall n, \;\;B\subset W,\]
and is the same for all $n$ and $v$ is the probability measure.
The system mapping $T(x,\xi)$ is assumed  continuous in $x$ and for every fixed $x\in X$, it is measurable in $\xi$. The initial condition, $x_0$, and the sequence of random vectors, $\xi_0,\xi_1,\ldots$, are assumed  independent. Let $W\times W=W^2$. Then, the two times composition of the stochastic dynamical system, denoted by $T^2: X\times W^2 \to X $, is given by
\[x_{n+2}=T(T(x_n,\xi_n),\xi_{n+1})=:T^2(x_n,\xi_n^{n+1}),\]
where $\xi_n^{n+1}\in W^2$. Since the sequence of random vectors $\{\xi_n\}$ is assumed  i.i.d, the probability measure on $W^2$ will simply be the product measure, $v\times v:=v^2$. Similarly, the $n$-times composition of a stochastic system (\ref{rds}), $T^n: X\times W^n\to X$, is denoted by $x_{n+1}=T^n(x_0,\xi_0^n)$, where $\xi_0^n\in W^n$ with probability measure $v^n$.

The basic object of study in our proposed approach to stochastic stability is a linear transfer, the Perron-Frobenius operator,  defined as follows:
\begin{definition}[Perron-Frobenius (P-F) operator] \label{def-PF}Let ${\cal M}(X)$ be the space of finite measures on $X$.
The Perron-Frobenius operator, $\mathbb{P}: {\cal M}(X)\to {\cal M}(X)$, corresponding to the stochastic dynamical system (\ref{rds}) is given by
\begin{eqnarray}
&[\mathbb{P}_T\mu](A)=\int_X \int_W \chi_A(T(x,y))dv(y)d\mu(x)\nonumber\\&=E_{\xi}\left[\mu(T_{\xi}^{-1}(A))\right],\label{pfequation}
\end{eqnarray}
 for $\mu\in {\cal M}(X)$, and $A\in {\cal B}(X)$, where $\chi_A(x)$ is an indicator function of set $A$.
\end{definition}
The expectations on $\xi$ are taken with respect to the probability measure, $v$, and $T^{-1}_{\xi}(A)$ is the inverse image of set $A$ under the mapping $T(x,\xi)$ for a fixed value of $\xi$, i.e.,
\[T^{-1}_{\xi}(A)=\{x: T(x,\xi)\in A\}.\]
Furthermore, if we denote the P-F operator corresponding to the dynamical system, $T_{\xi}:X\to X$, for a fixed value of $\xi$ as
\[[\mathbb{P}_{T_\xi}\mu](A)=\int_{X}\chi_A(T(x,\xi))d\mu(x)=\mu(T^{-1}_\xi(A)),\]
then, the P-F operator (\ref{pfequation}) can be written as
\[[\mathbb{P}_T\mu](A)=E_{\xi}[\mathbb{P}_{T_{\xi}}\mu](A).\]
The Koopman operator is dual to the P-F operator and  defined as follows:
\begin{definition}[Koopman operator] Let $h\in {\cal C}^0(X)$ be the space of continuous function. The Koopman operator, $\mathbb{U}_T :{\cal C}^0(X)\to {\cal C}^0(X)$,  corresponding to the stochastic system (\ref{rds}) is defined as follows:
\begin{eqnarray}
[\mathbb{U}_Th](x)=\int_W h(T(x,y))dv(y)=E_\xi[h(T(x,\xi))],
\end{eqnarray}
where the expectations are taken with respect to probability measure, $v$.
\end{definition}
Let $h$ be a measurable function and $\mu\in{\cal M}(X)$, define the inner product as
$\left<h,\mu\right>=\int_X hd\mu(x).$
With respect to this inner product, the Koopman and P-F operators are dual to each other. Using the inner product definition, we can write (\ref{pfequation}) as follows:
\[\left<\chi_A, \mathbb{P}_T\mu\right>=\left<\mathbb{U}_T\chi_A,\mu\right>.\]
Due to the linearity of the scalar product, this implies $\left<g_n,\mathbb{P}_T\mu\right>=\left<\mathbb{U}_T g_n, \mu\right>$,
where $g_n=\sum_{i=1}^n \alpha_i\chi_{A_i}$
is the sum of a simple function. Since, every measurable function, $h$, can be approximated by a sequence $\{g_n\}$ of simple functions, we obtain in the limit,
$\left<h,\mathbb{P}_T\mu\right>=\left<\mathbb{U}_T h, \mu\right>.$

\begin{assumption}\label{assume_equilibrium} We assume  $x=0$ is an equilibrium point of  system (\ref{rds}), i.e.,
$T(0,\xi_n)=0,\;\;\;\forall n,$
for any given sequence of random vectors $\{\xi_n\}$ taking values in set $W$.
\end{assumption}

\begin{assumption}[Local Stability]\label{assume_local} We assume  the trivial solution, $x=0$, is locally stochastic, asymptotically stable. In particular, we assume there exists a neighborhood ${\cal O}$ of $x=0$, such that for all $x_0\in \cal O$,
\[{\rm Prob}\{T^n(x_0,\xi_0^n)\in {\cal O}\}=1,\;\;\forall n\geq 0,\]
and
\[{\rm Prob}\{\lim_{n\to \infty} T^n(x_0,\xi_0^n)=0\}=1.\]
\end{assumption}
Assumption \ref{assume_equilibrium} is used in the decomposition of the P-F operator in section (\ref{section_decompose}) and  Assumption \ref{assume_local} is used in the proof of Theorem (\ref{theorem1}). In the following, we will use the notation $U(\epsilon)$ to denote the $\epsilon$ neighborhood of the origin for any positive value of $\epsilon>0$. We have $0\in U(\epsilon)\subset {\cal O}$.

We introduce the following definitions for stability of stochastic dynamical systems (\ref{rds}).
\begin{definition}[Almost everywhere (a.e.) almost sure stability]
The equilibrium point, $x=0$, is said to be almost everywhere, almost sure stable with respect to finite measure, $m\in {\cal M}(X)$, if for every $\delta(\epsilon)>0$, we have
\[m\{x\in X : Prob \{\lim_{n\to \infty} T^n(x,\xi_0^n)\neq 0\}\geq \delta\}=0.\]
\end{definition}

\begin{definition}[a.e. stochastic stable with geometric decay]\label{def_aeas_geometric} For any given $\epsilon>0$,  let $U(\epsilon)$ be the $\epsilon$ neighborhood of the equilibrium point, $x=0$. The equilibrium point, $x=0$, is said to be almost everywhere, almost sure stable with geometric decay with respect to finite measure,  $m\in {\cal M}(X)$,    if there exists $0<\alpha(\epsilon)<1$, $0<\beta<1$, and $K(\epsilon)<\infty$, such that
\[m \{x\in X: Prob \{T^n(x,\xi_0^n)\in B\}\geq \alpha^n\}\leq K \beta^n, \]
for all sets $B\in {\cal B}(X\setminus U(\epsilon))$, such that $m(B)>0$.
\end{definition}

We introduce the following definition of absolutely continuous and equivalent measures.
\begin{definition}  [Absolutely continuous measure] A measure $\mu$
is absolutely continuous with respect to another measure, $\vartheta$
denoted as $\mu\prec\vartheta$, if $\mu(B) = 0$ for all $B\in{\cal B}(X)$ with $\vartheta(B) =
0$.
\end{definition}
\begin{definition}[Equivalent measure] The two measures, $\mu$ and $\vartheta$, are
equivalent $(\mu \approx \vartheta)$ provided $\mu(B) = 0$, if and only if $\vartheta(B) = 0
$ for $B \in{\cal B}(X)$.
\end{definition}

\subsection{Decomposition of the P-F operator}\label{section_decompose}
Let $E=\{0\}$. Hence, $E^c=X\setminus E$. We write $T: E\cup E^c\times W\to X$. For any set $B\in {\cal B}(E^c)$, we can write

\begin{eqnarray}&[\mathbb{P}_T\mu](B)=\int_X \int_W \chi_B(T(x,y))dv(y)d\mu(x)\nonumber\\&=\int_{E^c} \int_W \chi_B(T(x,y))dv(y)d\mu(x).\end{eqnarray}
This is because  $T(x,\xi)\in B$ implies  $x\notin E$.
Since  set $E$ is invariant, we define the restriction of the P-F operator on the complement set $E^c$. Thus, we can define the restriction of the P-F operator on the measure space ${\cal M}(E^c)$ as follows:
\[[\mathbb{P}_1\mu](B)=\int_{E^c} \int_W \chi_B(T(x,y))dv(y)d\mu(x),\]
for any set $B\in{\cal B}(E^c)$ and $\mu\in {\cal M}(E^c)$.

 Next, the restriction $T:E\times W\rightarrow E$ can also be used
to define a P-F operator denoted by
\begin{equation*}
[\mathbb{P}_0\mu](B)=\int_{B} \chi_B(T(x,y)) dv(y) d\mu(x),
\end{equation*}
where $\mu\in{\cal M}(E)$ and $B\subset {\cal B}(E)$.

The above considerations suggest a representation of the P-F
operator, $\mathbb{P}$, in terms of $\mathbb{P}_0$ and $\mathbb{P}_1$.
Indeed, this is the case, if one considers a splitting of the measured
space,
\begin{equation}
{\cal M}(X)={\cal M}_0\oplus{\cal M}_1, \label{eq:splitM}
\end{equation}
where ${\cal M}_0 := {\cal M}(E)$, ${\cal M}_1 := {\cal
M}(E^c)$, and $\oplus$ stands for direct sum.


Then it follows   the splitting defined by
Eq.~(\ref{eq:splitM}), the P-F operator has a lower-triangular
matrix representation given by
\begin{equation}
\mathbb{P}=\left[ \begin{array}{cc} \mathbb{P}_0 & 0 \\
\times & \mathbb{P}_1 \end{array} \right]. \label{eq:splitP}
\end{equation}


\section{Lyapunov measure and stochastic stability}\label{section_main}
We begin the main results section with the following Lemma.

\begin{lemma}\label{lemma1}
Let \[\xi_0^n=\{\xi_0,\ldots, \xi_n\}\in  \underbrace{W\times\ldots\times W}_{n}=:W^n\] and $F(x,\xi_0^n):=T^n(x,\xi_0^n): X\times W^n\to X$ be the notation for the $n$ times composition of the map $T:X\times W\to X$. Then, the Perron-Frobenius operator, $\mathbb{P}_F: {\cal M}(X)\to {\cal M}(X)$, corresponding to system mapping $F$ is given by
\[\mathbb{P}_F=\underbrace {\mathbb{P}_T \mathbb{P}_T\ldots\mathbb{P}_T}_{n}=: \mathbb{P}_T^n.\]
\end{lemma}
\begin{proof}
Let $\vartheta(A)=[\mathbb{P}_T\mu](A)$, it then follows from  Definition \ref{def-PF} for the P-F operator that
\begin{eqnarray}
[\mathbb{P}_T \mu](A)=\vartheta(A)=E_{\xi_0}\left[\mu(T_{\xi_0}^{-1}(A))\right],\label{eqn1}
\end{eqnarray}
where expectation is taken with respect to probability measure $v$ and $T_{\xi_0}^{-1}(A)$ is the inverse image of set $A$ under the mapping  $T(x,\xi_0)$ for fixed values of $\xi_0$, i.e.,
\[T_{\xi_0}^{-1}(A)=\{x\in X: T(x,\xi_0)\in A\}.\]
Hence,
\[[\mathbb{P}_T^2\mu](B)=[\mathbb{P}_T \mathbb{P}_T \mu](B)=[\mathbb{P}_T \vartheta](A)=E_{\xi_1}[\vartheta (T_{\xi_1}^{-1}(B))].\]
Following (\ref{eqn1}) and defining $T_{\xi_1}^{-1}(B)=A$, we obtain
\[[\mathbb{P}_T^2\mu](B)=E_{\xi_1}\left[E_{\xi_0}\left[\mu(T^{-1}_{\xi_0}(T^{-1}_{\xi_1}(B)))\right]\right].\]
Since $\xi_0$ and $\xi_1$ are independent with the same probability distribution $v$, we obtain
\begin{eqnarray}&[\mathbb{P}_T^2\mu](B)=E_{\xi_0^1}\left[\mu(T^{-1}_{\xi_0}(T^{-1}_{\xi_1}(B)))\right]\nonumber\\&=\int_{W^2} \int_X \chi_B(T(T(x,y_0),y_1))dv(y_0)dv(y_1)d\mu(x)\nonumber\\&=[\mathbb{P}_{T^2(x,\xi_0^1)} \mu](B).\end{eqnarray}
The main results of this lemma then follow from induction.
\end{proof}

Using the lower triangular structure of the P-F operator in Eq. (\ref{eq:splitP}), one can write the $\mathbb{P}_T^n$ as follows:
\begin{eqnarray}
\mathbb{P}_T^n=\left[ \begin{array}{cc}\mathbb{P}_0^n&0\\\times & \mathbb{P}_1^n\end{array}\right].
\end{eqnarray}

We now state the first main results of the paper on the stochastic stability expressed in terms of asymptotic behavior of $\mathbb{P}^n_1$.
\begin{theorem}\label{theorem1} The equilibrium point, $x=0$, for system (\ref{rds}) is almost everywhere, almost sure stable with respect to finite measure, $m\in {\cal M}(X)$, if
\[\lim_{n\to \infty}[\mathbb{P}^n_1 m](B)=0,\]
for every set $B\in{\cal B}(X\setminus U(\epsilon))$, such that $m(B)>0$. $U(\epsilon)$ is the $\epsilon$ neighborhood of the equilibrium point, $x=0$, for any given $\epsilon>0$.
\end{theorem}
\begin{proof}
For any given $\delta(\epsilon)>0$, consider the following sets,
\[S_n=\{x\in X: Prob(T^k(x,\xi_0^k)\in X\setminus U(\epsilon))\geq \delta,\;{\rm for \;some \;}k>n\}\] and
\[S=\cap_{n=1}^\infty S_n.\]
So, set $S$ consists of points with probability larger than $\delta$ to end up in set $X\setminus U(\epsilon)$.  Now, since $\epsilon$ is arbitrary small and from the local stability property of equilibrium point, we have $0\in U(\epsilon)\subset {\cal O}$. The proof of this theorem follows, if we show  $m(S)=0$. Let $\tilde S=S \cap (X\setminus {\cal O})$. Then, from the property of the local neighborhood,  $m(S)=m(\tilde S)$. From the construction of set $S$, it follows  $x\in S$, if and only if $Prob(T(x,\xi)\in S)=1$. Hence,
\begin{eqnarray}[\mathbb{P}_T m](S)=\int_X Prob(\chi_S(T(x,\xi)))dm(x)=m(S).\label{ee1}
\end{eqnarray}
Now, $\tilde S\subset S$ with $m(\tilde S)=m(S)$. Since $T$ is continuous in $x$ and measurable in $\xi$, we have $[\mathbb{P}_1 m](\tilde S)=[\mathbb{P}_1 m]( S)$. Using (\ref{ee1}), we obtain $[\mathbb{P}_1 m](\tilde S)=m(\tilde S)$. $\tilde S\subset X$ lies outside some local neighborhood of $x=0$.
However, $\lim_{n\to \infty}[\mathbb{P}_1^n m](B)=0$ for any set $B\in {\cal B}(X\setminus U(\epsilon))$ and, in particular, for $B=\tilde S$.  Hence, we have $m(\tilde S)=m(S)=0$.
\end{proof}

\begin{theorem}\label{theorem_aegeometric} The $x=0$ solution for  system (\ref{rds}) is a.e. stochastic stable with geometric decay with respect to finite measure, $m\in {\cal M}(X)$, if and only if there exists a positive constant, $K(\epsilon)$, and $0<\beta<1$, such that
\[[\mathbb{P}^n_1 m](B)\leq K \beta^n, \;\;\;\forall n\geq 0\]
for every set $B\in {\cal B}(X\setminus U(\epsilon))$, such that $m(B)>0$.
\end{theorem}

\begin{proof}
We first prove the sufficient part.
\[[\mathbb{P}^n_1 m](B)=\int_{X}\int_{W^n} \chi_B(T^n(x,\xi_0^n))d v_0^n(\xi_0^n)dm(x)\leq K\beta^n.\]
\begin{eqnarray}&[\mathbb{P}^n_1 m](B)=\int_{X}\int_{W^n} \chi_B(T^n(x,\xi_0^n))d v_0^n(\xi_0^n)dm(x)\nonumber\\&=\int_X Prob \{T^n(x,\xi_0^n)\in B\}dm(x).\end{eqnarray}
Choose $0<\alpha<1$, such that $\frac{\beta}{\alpha}<1$.
Let \[S_n=\{x\in X : Prob\{T^n(x,\xi_0^n)\in B\}\leq \alpha^n\}\] and \[\bar S_n=\{x\in X:  Prob\{T^n(x,\xi_0^n)\in B\}\geq  \alpha^n\}.\]
Then,
\begin{eqnarray*}
[\mathbb{P}^n_1 m](B)&=&\int_{S_n} Prob \{T^n(x,\xi_0^n)\in B\} dm(x)\nonumber\\&+&\int_{\bar S_n} Prob \{T^n(x,\xi_0^n)\in B\}dm(x).
\end{eqnarray*}
Since integrals over both  sets are positive, we have
\[[\mathbb{P}^n_1 m](B)\geq \int_{\bar S_n}Prob \{T^n(x,\xi_0^n)\in B\}dm(x)\geq \alpha^n m(\bar S_n).\]

Hence, we have,
\[m(\bar S_n)\leq K \frac{\beta^n}{\alpha^n},\]
which is equivalent to
\[m\{x\in X: Prob\{T^n(x,\xi_0^n)\in B\}\geq\alpha^n\}\leq K\bar \beta^n.\]
For the necessary part, we assume  the system is a.e. stochastic stable with geometric decay (Definition \ref{def_aeas_geometric}). Construct the sets, $S_n$ and $\bar S_n$, from the sufficiency part of the proof. The parameter, $\alpha$, for this construction now comes from the stability definition (Definition \ref{def_aeas_geometric}). We have
\begin{eqnarray*}
[\mathbb{P}^n_1 m](B)&=&\int_{S_n} Prob \{T^n(x,\xi_0^n)\in B\} dm(x)\nonumber\\&+&\int_{\bar S_n} Prob \{T^n(x,\xi_0^n)\in B\}dm(x).
\end{eqnarray*}
Hence,
\[[\mathbb{P}^n_1 m](B)\leq \alpha^n m(S_n)+m(\bar S_n)\leq \alpha^n m(S_n)+K\beta^n. \]
Now, since $X$ is assumed  compact, we have $m(S_n)<M$ for some $M<\infty$. Hence, we have
\[[\mathbb{P}^n_1 m](B)\leq  \bar K \bar \beta^n,\]
where $\bar K=\max(M,K)$ and $\bar \beta=\max(\beta,\alpha)<1$.
\end{proof}

Following is the definition of Lyapunov measure introduced for stability verification of a stochastic system.
\begin{definition}[Lyapunov measure]\label{definition_Lyameas} A Lyapunov measure, $\bar \mu\in{\cal M}(X\setminus U(\epsilon))$, is defined as any positive measure  finite outside the $\epsilon$ neighborhood of equilibrium point and satisfies
\begin{eqnarray}
[\mathbb{P}_1 \bar \mu](B)< \gamma \bar \mu(B) \label{lya_meas}
\end{eqnarray}
for $0<\gamma\leq 1$ and for all sets $B\in {\cal B}(X\setminus U(\epsilon))$.
\end{definition}

\begin{theorem} Consider the stochastic dynamical system (\ref{rds}) with $x=0$ as a locally stable, equilibrium point. Assume there exists a Lyapunov measure, $\bar \mu$, satisfying Eq. (\ref{lya_meas}) with $\gamma<1$. Then,
\begin{enumerate}
\item $x=0$ is almost everywhere almost sure stable with respect to finite measure, $m$, which is absolutely continuous with respect to the Lyapunov measure $\bar \mu$.
    \item $x=0$ is  almost everywhere stochastic stable with geometric decay with respect to measure any finite measure, $m\prec\kappa \bar \mu$, for some constant, $\kappa>0$.
\end{enumerate}
\end{theorem}
\begin{proof} 1) Using the definition of Lyapunov measure with $\gamma<1$, we obtain
\[[\mathbb{P}^n_T \bar \mu](B)\leq \gamma^n \bar \mu(B)\implies \lim_{n\to \infty} [\mathbb{P}_T^n \bar \mu](B)=0.\]
Now,  since $m\prec\bar \mu$, we have
$ \lim_{n\to \infty} [\mathbb{P}_T^n m](B)=0.$
The proof then follows by applying the results from Theorem \ref{theorem1}.\\
2) We have
\[[\mathbb{P}_T^n m](B)\leq \kappa [\mathbb{P}_T^n \bar \mu](B)\leq\kappa \gamma^n \bar \mu(B)\leq K(\epsilon)\gamma^n, \]
where $K(\epsilon)=\kappa \bar \mu(X\setminus U(\epsilon))$, which, by  definition of the Lyapunov measure, is finite. The proof then follows by applying the results from Theorem \ref{theorem_aegeometric}.
\end{proof}

The following theorem provides for the construction of the Lyapunov measure as an infinite series involving the P-F operator.

\begin{theorem} \label{theorem_measure_series} Let the equilibrium point, $x=0$, be almost everywhere, stochastic stable with geometric decay with respect to measure, $m$. Then, there exists a Lyapunov measure (Definition \ref{definition_Lyameas}) $\bar \mu$ with $\gamma<1$. Furthermore, the Lyapunov measure is equivalent to measure $m$ (i.e., $\bar \mu\approx m$) and the Lyapunov measure dominates measure $m$ (i.e., $\bar \mu(B)\geq m(B))$.
\end{theorem}
\begin{proof}
Following the results of Theorem \ref{theorem_aegeometric}, we know  there exists a positive constant, $K$, and $\beta<1$, such that
\[[\mathbb{P}_1^n m](B)\leq K \beta^n.\]
Let $\beta=\beta_1\beta_2$, such that $\beta_1<1$ and $\beta_2<1$. Hence, we have
\[\alpha^n[\mathbb{P}_1^n m](B)\leq K \beta_1^n\]
with $\alpha=\frac{1}{\beta_2}>1$. Now, construct the Lyapunov measure as follows:
\begin{eqnarray}&\bar \mu(B)=\left(m+\alpha  [\mathbb{P}_1 m]+\alpha^2+\ldots\right)(B)\nonumber\\&=\sum_{k=0}^{\infty}\alpha^k [\mathbb{P}_1^k m](B).\label{infiniteseries}\end{eqnarray}
The above infinite series is well-defined and converges because $\alpha^n[\mathbb{P}_1^n m](B)\leq K \beta_1^n$. Multiplying both sides of (\ref{infiniteseries}) with $I-\alpha \mathbb{P}_1$, we obtain
\[\alpha [\mathbb{P}_1\bar \mu](B)\leq \bar \mu(B)\implies[\mathbb{P}_1\bar \mu](B)\leq \beta_2 \bar \mu(B). \]
 Allowing $\gamma=\beta_2$, we satisfy the requirements for the Lyapunov measure. The equivalence of measure, $m$, the Lyapunov measure, $\bar \mu$, and the dominance of the Lyapunov measure   follows from the construction of the Lyapunov measure as a infinite series formula Eq. (\ref{infiniteseries}).
\end{proof}

\section{Koopman operator and Lyapunov function}\label{section_koopman}

In this section, we describe the connection between the Koopman operator and Lyapunov function for a stochastic dynamical system.
\begin{definition}[$p^{th}$ Moment Stability] The equilibrium solution $x=0$ is said  $p^{th}$ moment exponentially stable with $p\in \mathbb{Z}^+$, if there exist a positive constants, $K<\infty$, and $\beta<1$, such that
\[E_{\xi_0^n}[\parallel x_{n+1}\parallel^p ]\leq K\beta^n \parallel x_0\parallel^p,\;\;\;\forall n\geq 0\]
for all initial conditions $x_0\in X$.
\end{definition}
The results of the following theorem are not new \cite{Hasminskii_book, mao_book}. However, the proposed construction of the Lyapunov function in terms of the Koopman operator is new. Furthermore, the connection between the Koopman operator and Lyapunov function also brings out clearly the dual nature of the Lyapunov function and the Lyapunov measure for stochastic stability verification.
\begin{theorem}\label{Theorem_Lyap} The equilibrium solution, $x=0$, is $p^{th}$ moment exponentially stable, if and only if there exists a nonnegative function $V:X\to \mathbb{R}^{+}$ satisfying
\[a\parallel x \parallel^p\leq V(x)\leq b\parallel x \parallel^p,\;\;\;\; [\mathbb{U}_TV](x)\leq c V(x),\]
where $a,b,c$ are positive constants with $c<1$. Furthermore, $V$ can be expressed in terms of  resolvent of the  Koopman operator as follows:
\[V(x)=(I-\mathbb{U}_T)^{-1}f(x),\]
where $f(x)=\parallel x\parallel^p$.
\end{theorem}
\begin{proof} We first prove the sufficient part. We have
\begin{eqnarray}&E_{\xi_0^n}[\parallel x_{n+1}\parallel^p]\leq \frac{1}{a}E_{\xi_0^n}[V(x_{n+1})]=\frac{1}{a}E_{\xi_0^{n-1}}\left[[\mathbb{U}_T V](x_n)\right]\nonumber\\
& \leq \frac{c}{a}E_{\xi_0^{n-1}}[V(x_n)]\leq \frac{c^n}{a}V(x_0)\leq \frac{b}{a}c^n \parallel x_0\parallel^p.
\end{eqnarray}
Let, $K=\frac{b}{a}$ and $c=\beta$. Thus, we obtain the desired condition for the  $p^{th}$ moment exponential stability. For the necessary part,
let
\begin{eqnarray}&V_N(x_0)=\sum_{k=0}^N E_{\xi_0^k}[f( x_{k+1})]=\sum_{k=0}^NE_{\xi_0^{k-1}}[[\mathbb{U}_T f](x_k)]\nonumber\\&=\sum_{k=0}^N [\mathbb{U}_T^k f](x_0).\end{eqnarray}
The uniform bound on $V_N(x)$ follows from $p^{th}$ moment exponential stability. Hence, $V(x)=\lim_{N\to \infty} V_N(x)$ is well defined. Furthermore,
\begin{eqnarray}
V(x)=\lim_{N\to \infty}\sum_{k=0}^N [\mathbb{U}_T^k f](x_0)=(I-\mathbb{U}_T)^{-1}f. \label{infiniteseries_koopman}
\end{eqnarray}
The bounds on the function, $V$, and the inequality, $[\mathbb{U}_T V](x)\leq cV(x)$, follow from the construction of the function, V.
\end{proof}

The P-F operator and Koopman operator are shown  dual to each other. In Eqs. (\ref{infiniteseries}) and (\ref{infiniteseries_koopman}), the Lyapunov measure and the Lyapunov function are expressed in terms of an infinite series involving the P-F and Koopman operators, respectively. Using the duality relationship between the P-F and Koopman operators, it then follows  the Lyapunov function and the Lyapunov measure used for verifying stochastic stability are dual to each other. In the following section, we establish a precise connection between the two stochastic, stability verification tools.

\section{Relation between Lyapunov measure and function}\label{section_relation}
In this section, we relate Lyapunov measure and Lyapunov function. We impose additional assumptions, the system mapping (\ref{rds}), i.e., $T:X\times W\to X$. We assume  the system mapping $T$ is $C^1$ invertible diffeomorphism with respect to state variable, $x$, for any fixed value of noise parameter,  $\xi$. For the diffeomorphism, we define
\[J_\xi^{-1}(x)=\left|\frac{d T_\xi^{-1}}{dx}(x)\right|,\]
where $|\cdot|$ stands for the determinant.
\begin{lemma}\label{lemma}
Let $\mathbb{P}_T$ be the P-F operator for the system mapping $T:X\times W\to X$. Then,
\[d[\mathbb{P}_Tm](x)=E_{\xi}[J^{-1}_\xi(x)]dm(x).\]
If $\rho(x)$ is the density of an absolutely continuous measure, $\mu$, with respect to the Lebesgue measure, $m$, i.e., $d\mu(x)=\rho(x)dm(x)$, then
\[d[\mathbb{P}_T\mu](x)=E_{\xi}[J_\xi^{-1}(x)\rho(T_\xi^{-1}(x))]dm(x).\]
\end{lemma}
\begin{proof}
\[[\mathbb{P}_T m](B) =\int_W \int_X \chi_B(T_y(x))dm(x)dv(y)\]
\[=\int_W \int_X \chi_B(x)dm(T_y^{-1}(x))dv(y).\]
\[\int_X \chi_B(x)\int_W J_y^{-1}(x)dm(x)dv(y)=\int_X \chi_B(x)E[J_\xi^{-1}(x)]dm(x).\]
\[d[\mathbb{P}_T m]=E_\xi[J_\xi^{-1}(x)]dm(x).\]
\[d \bar \mu(x)=\rho(x)dm(x).\]
\[[\mathbb{P}_T \bar \mu](B)=\int_X\int_W \chi_B(T_y(x))dv(y)\rho(x)dm(x)\]
\[=\int_X\int_W \chi_B(x)dv(y)\rho(T_y^{-1}(x))J_y^{-1}(x)dm(x).\]
\[d[\mathbb{P}_T \bar \mu]=E[\rho(T_\xi^{-1}(x))J_\xi^{-1}(x)] dm(x).\]
%
\end{proof}

We have the following theorem connecting the Lyapunov measure and the Lyapunov function for a stochastic system.
\begin{theorem} Let $J_{\xi}(x)<\Delta<1$ for all $\xi\in W$ and $x\in X$.
1)
Let the equilibrium point, $x=0$, be almost everywhere, stochastic stable with geometric decay with respect to the Lebesgue measure, $m$. Assume  the Lyapunov measure is absolutely continuous with respect to the Lebesgue measure, $m$,  with density function $\rho(x)$, i.e., $d\bar \mu(x)=\rho(x)dm(x)$. Hence,  (following Theorem \ref{theorem_measure_series}), there exists a Lyapunov measure satisfying
    \[d\bar \mu(x)-\alpha d[\mathbb{P}_1 \mu](x)=g(x)dm(x).\]
    Furthermore, the density function corresponding to $\rho(x)$ satisfies

    \begin{eqnarray}
    a_1\parallel x\parallel^{-p}\leq \rho(x)\leq a_2\parallel x\parallel^{-p}.\label{radially_unbounded}
    \end{eqnarray}
    Then, the $x=0$ solution is $p^{th}$ exponentially stable with Lyapunov function $V$ obtained as
    $V(x)=\rho(x)^{-1}$.\\
2)
    Let  $x=0$ be the $p^{th}$ moment exponentially stable with the Lyapunov function $V(x)$ satisfying
    $V(x)<\beta E_{\xi}[J_\xi(x)V(T_{\xi}^{-1}(x))]$
    for some $\beta<1$.
    Then, the measure,
    \begin{eqnarray}\bar \mu(B)=\int_B \frac{1}{V^\gamma(x)}dm(x),
    \label{lya}\end{eqnarray}
    is a Lyapunov measure satisfying
    $E_{\xi}[\bar\mu(T_\xi^{-1}(B))]<\kappa \bar \mu(B)$
    for some $\kappa<1$ and for all $B\in{\cal M}(X\setminus U(\epsilon))$ with $m(B)>0$. $\gamma\geq 1$ is a suitable constant chosen, such that $\frac{1}{V^\gamma}$ is integrable.
\end{theorem}
\begin{proof}
1)
\begin{eqnarray}
&d\bar \mu(x)-\alpha d[\mathbb{P}_1 \mu](x)=g(x)dm(x)\nonumber\\&=V^{-1}(x)dm(x)-\alpha E_{\xi}[J_\xi^{-1}(x)V^{-1}(T_\xi^{-1}(x))]dm(x).
\end{eqnarray}
Hence, we have
$V^{-1}(x)-\alpha E_{\xi}[J_\xi^{-1}(x)V^{-1}(T_\xi^{-1}(x))]=g(x)\geq 0.$
Now, since $J_{\xi}(x)<1$ or $J^{-1}_\xi(x)>1$ for all values of $\xi\in W$, we have
\begin{eqnarray}
\frac{1}{V(x)}\geq \alpha E_\xi\left[\frac{1}{V(T_\xi^{-1}(x))}\right].\label{e1}
\end{eqnarray}
Now, using Holder's inequality, we have
\[1=E_\xi\left[\frac{V(T^{-1}_\xi(x))}{V(T_\xi^{-1}(x))}\right]\leq E_\xi \left[\frac{1}{V(T_\xi^{-1}(x))}\right]E_\xi[V(T^{-1}_\xi(x))]. \]
\begin{eqnarray}
E_\xi \left[\frac{1}{V(T_\xi^{-1}(x))}\right]\geq \frac{1}{E_\xi[V(T^{-1}_\xi(x))]}.\label{e2}
\end{eqnarray}

Combining (\ref{e1}) and (\ref{e2}), we obtain
\begin{eqnarray}  V(x)\leq \bar \alpha E_{\xi}[V(T_\xi^{-1}(x))],\;\;\;\bar\alpha=\frac{1}{\alpha}<1.\label{eqn2}
\end{eqnarray}
Since $V(x)=\rho^{-1}(x)$ is assumed radially unbounded (Eq. (\ref{radially_unbounded})), then using the results from Theorem (\ref{Theorem_Lyap}) and  inequality (\ref{eqn2}), we achieve the desired results,  the $x=0$ solution
 is the $p^{th}$ moment exponentially stable.\\
2) We have $V(x)^\gamma< \beta^\gamma E_{\xi}[J_\xi(x) V(T_\xi^{-1}(x))]^\gamma.$
 Using the Holder inequality, we obtain
\[V(x)^\gamma< \bar\beta E_\xi [ J^\gamma_\xi(x) V^\gamma (T_\xi^{-1}(x))],\]
where $\bar \beta=\beta^\gamma<1$.

\[\frac{\bar\beta}{V(x)^\gamma}> \frac{1}{ E_\xi [ J^\gamma_\xi(x) V^\gamma (T_\xi^{-1}(x))]}.\]

Now, since $J_\xi<\Delta<1$ is uniformly bounded, there exists a $\gamma>1$, sufficiently large,  such that
\[\frac{\bar\beta}{V(x)^\gamma}>\frac{1}{ E_\xi [ J^\gamma_\xi(x) V^\gamma (T_\xi^{-1}(x))]}\geq E_\xi\left[\frac{1}{J_\xi(x)V^\gamma (T_\xi^{-1}(x))}\right].\]

Integrating over set $B\in {\cal B}(X\setminus U(\epsilon))$, we obtain

\[\bar \beta\int_B V^{-\gamma}(x)dm(x)\geq \int_B J_\xi^{-1}(x)V^{-\gamma}(T_\xi^{-1}(x))dm(x).\]
Using the results from Lemma (\ref{lemma}) and (\ref{lya}), we  obtain the desired results by letting $\kappa=\bar \beta$, i.e.,
\[[\mathbb{P}_1\bar \mu](B)\leq \kappa \bar \mu(B).\]

\end{proof}

\section{Discretization of the P-F operator}\label{section_discrete}
In this section, we will discuss the set-oriented numerical methods proposed for the finite dimensional approximation of the P-F operator. For the finite dimensional approximation of the P-F operator, we will make the following assumption on the stochastic process.
\begin{assumption} \label{finite_noise} We assume the random vector $\xi_n$ takes finitely many vector values $w_1,\ldots, w_Q$ with the following probabilities
\[{\rm Prob} \{\xi_n=w_\ell\}=p_\ell \;\;\;\forall n ,\;\;\;\ell=1,\ldots, Q.\]
\end{assumption}

For the finite dimension approximation of the P-F operator, we consider the finite partition of the state space, $X$, as follows:
\[{\cal X}=\{D_1,\ldots, D_L\},\]
where $\cup_{j} D_j =X$.  The infinite dimensional measure, $\mu$, is approximated by ascribing a real number $\mu_j$ to each cell of the partition, $\cal X$. In this way, the infinite dimensional measure, $\mu$, is approximated by a finite dimensional vector, $\mu\in \mathbb{R}^L$. The finite dimensional approximation of the P-F operator will arise as a Markov matrix of size $L\times L$ on the finite dimensional vector space. Using Assumption \ref{finite_noise}, the stochastic dynamical system, $T(x,\xi):X\times W\to X$, is parameterized by finitely many values that the random variable takes. In particular, for each fixed value of a random variable, i.e., $\xi_n=w_k$, we have $T(x,w_k)$ for $k=1,\ldots, Q$. For notational convenience, we will also write the mapping for each fixed value of random variable as  $T(x,w_k)=:T_{w_k}(x)$. Hence, we have $T_{w_k}: X\to X$ for $k=1,\ldots, Q$. With the finite parameterizations in the noise space, the finite dimensional approximation of the P-F operator follows exactly along the lines of a deterministic setting.
In particular, corresponding to a vector,
$\mu=(\mu_1,\cdots,\mu_L)\in\mathbb{R}^L$, define a measure on $X$
as
\begin{equation*}
d\mu(x)=\sum_{i=1}^L \mu_i \kappa_{i}(x)
\frac{dm(x)}{m(D_i)},\label{measure_resp}
\end{equation*}
where $m$ is the Lebesgue measure and $\kappa_j$ denotes the
indicator function with support on $D_j$. The approximation, denoted
by $P$, is now obtained as
\begin{eqnarray*}
\nu_j&=&[\mathbb{P}\mu](D_j)\nonumber\\&=&\sum_{\ell=1}^Q p_\ell \sum_{i=1}^L \int_{D_i}
\delta_{T_{w_\ell}(x)}(D_j)\mu_i\frac{dm(x)}{m(D_i)}\nonumber\\& =&\sum_{\ell=1}^Q p_\ell\sum_{i=1}^L \mu_i
{P}^{w_\ell}_{ij}= \sum_{i=1}^L\mu_i\sum_{\ell=1}^Q p_\ell
{P}^{w_\ell}_{ij},
\end{eqnarray*}
where
\begin{equation}
{P}^{w_\ell}_{ij}=\frac{m(T_{w_\ell}^{-1}(D_j)\cap D_i)}{m(D_i)}, \label{eq:pij}
\end{equation}
$m$  the Lebesgue measure.  The resulting matrix is
non-negative and because $T_{w_\ell}:D_i\rightarrow X$ for $\ell=1,\ldots,Q$,
\begin{equation*} \sum_{j=1}^L P^{w_\ell}_{ij}=1,\;\;\ell=1,\ldots, Q,
\end{equation*}
i.e., $P$ is a Markov or a row-stochastic matrix.

Computationally, several short-term trajectories are used to
compute the individual entries ${P}^{w_\ell}_{ij}$.  The mapping $T_{w_\ell}$ is
used to {\em transport} $M$ ``initial conditions" chosen to be
uniformly distributed within a set $D_i$.  The entry ${P}^{w_\ell}_{ij}$ is
then approximated by the fraction of initial conditions
in  box, $D_j$, after one iterate of the mapping.  In the
remainder of the paper, the notation of this section is used,
whereby $P$ represents the finite-dimensional Markov matrix
corresponding to the infinite dimensional P-F operator
$\mathbb{P}$. The finite dimensional approximation of the P-F operator can be used to study the approximated dynamics of the stochastic dynamical system. In particular, the eigenvector with eigenvalue one of the matrix $P$, i.e.,
\[\mu P=\mu,\;\;\mu_i\geq 0,\;\;\sum_{i=1}^L \mu_i=1,\]
captures the steady state dynamics of the system.

\section{Stability in finite dimension}\label{section_coarse}
The finite dimensional approximation of the P-F operator presented in the previous section can be decomposed into the lower triangular form similar to the decomposition of the infinite dimensional P-F operator in Eq. (\ref{eq:splitP}). With no loss of generality, we  assume  the equilibrium point at the origin is contained inside the cell, $D_1$, and the partition is sufficiently fine so  $D_1\subset {\cal O}$, where ${\cal O}$ is the local neighborhood of the stable equilibrium point. We decompose $\cal X$ into two complimentary  partition as follows,
\begin{eqnarray}
{\cal X}_0 = \{ D_1 \},\;\;\;
{\cal X}_1 = \{ D_{2},...,D_L \},\label{eq:partx1}
\end{eqnarray}
with domains $X_0=D_1$ and $X_1=\cup_{k=2}^{L}D_k$.
Since the local stable equilibrium point at the origin is contained in the cell, $D_1$, there exists a left  eigenvector with eigenvalue one to the Markov matrix, $P$, of the form $\mu_0=(1,0\ldots,0)$, i.e., $\mu_0=\mu_0 P$. The existence of the eigenvector can be used to decompose the Markov matrix, $P$, in the following upper triangular structure. Let $M\cong \mathbb{R}^L, M_0\cong \mathbb{R}$, and $M_1\cong \mathbb{R}^{L-1}$ denote the finite dimensional measure space associated with the partition ${\cal X}, {\cal X}_0$, and ${\cal X}_1$, respectively and as defined in (\ref{eq:partx1}). Then, for the splitting $M=M_0\oplus M_1$, the $P$ matrix has the following lower triangular representation.
\[P=\begin{pmatrix}P_0&0\\\times & P_1\end{pmatrix},\]
where $P_0=1$ and maps $P_0: M_0\rightarrow M_0$  and $P_1: M_1\rightarrow M_1$ is the sub-Markov matrix
with row sum less than or equal to one. We refer the readers to \cite{Vaidya_TAC} for details on the decomposition in finite dimension. Our goal is to study stability with respect to initial condition starting from partition ${\cal X}_1$ using the sub-Markov matrix, $P_1$.
In a discrete finite dimensional setting,  stability is expressed in terms of the
transient property of the stochastic matrix $P_1$.
\begin{definition}[Transient states] A sub-Markov matrix $P_1$
has only transient states, if $P_1^n \rightarrow 0$,
element-wise, as $n\rightarrow \infty$.\label{def:transp1}
\end{definition}
Transience of $P_1$ is shown to imply a weaker notion of stability referred to as coarse stochastic stability and  defined as follows.
\begin{definition}[Stochastic Coarse Stability]\label{def_coarsestability} Consider the finite partition, ${\cal X}_1$, of the complement set, $X_1=X\setminus X_0$. The equilibrium point, $x=0$, is said to be stochastic coarse stable with respect to the initial condition in $X_1$, if for  an attractor set, $B\subset U\subset X_1$ there exists no subpartition ${\cal S}=\{D_{s_1},\ldots,D_{s_l}\}$ in ${\cal X}_1$ with domain $S=\cup_{k=1}^l D_{s_k}$ such that $B\subset S\subset U$ and for all $x\in S$ $Prob\{ T(x,\xi)\in S\}=1$
.
\end{definition}

For typical partitions, coarse stability means stability modulo
attractor sets, $B$, with domain of attraction, $U$, smaller than the
size of cells within the partition. In the infinite-dimensional
limit, where the cell size (measure) goes to zero, one obtains
stability modulo attractor sets with measure $0$ domain of
attraction.

\begin{theorem} Let the equilibrium point, $x=0\in X_0\subset X$, with local domain of attraction contained inside $X_0$, $\mathbb{P}_1$ is the sub-Markov operator on ${\cal M}(E^c)$. $P_1$ be its finite dimensional approximation obtained with respect to the partition ${\cal X}_1$ of the complement set $X_1=X\setminus X_0$. For this, we have the following
\begin{enumerate}
\item Suppose a Lyapunov measure, $\bar \mu$, exists such that
\begin{eqnarray}[\mathbb{P}_1\bar \mu](B) < \bar \mu(B)\label{eq:eqnll22}
\end{eqnarray}
for all $B \subset {\cal B}(X_1)$, and additionally, $\bar \mu\equiv m$, the Lebesgue
measure. Then, the finite-dimensional approximation, $P_1$,
is transient.
\item Suppose $P_1$ is transient, then the equilibrium point, $x=0$, is coarse stable with respect
to the initial conditions in $X_1$.
\end{enumerate}
\end{theorem}
\begin{proof}
Before stating the proof, we claim for any two sets, $S_1$ and
$S$, such that $S_1\subset S$, if $\mu\approx m$, then
\begin{equation}
\mu(S_1)=\mu(S)\;\;\;\;\Longleftrightarrow\;\;\;\;m(S_1)=m(S).\label{claim_lemma}
\end{equation}
Denote $S_1^c :=S \setminus S_1$ as the complement set. We have,
$\mu(S_1)=\mu(S)$ implies $\mu(S_1^c)=0$ which, in turn, implies
$m(S_1^c)=0$ and, thus, $m(S_1)=m(S)$.

  1. We first present a proof
for the simplest case where the partition, ${\cal X}_1$, consists of
precisely one cell, i.e., ${\cal X}_1=\{D_L\}$. In this case,
$P_1\in[0,1]$ is a scalar given by
\begin{equation}
P_1 = \sum_{k=1}^Q p_k \frac{m(T_{w_k}^{-1}(D_L)\cap D_L)}{m(D_L)},\label{eq:p1inthm1}
\end{equation}
where $m$ is the Lebesgue measure. We need to show $P_1<1$.
Denote,
\begin{equation}
S_k=\{x\in D_L: \;T_{w_k}(x)\in D_L\},\;\;k=1,\ldots, Q.\label{eq:sands1}
\end{equation}
Clearly, $S_k\subset D_L$ and the existence of Lyapunov measure
$\bar{\mu}$ satisfying Eq.~(\ref{eq:eqnll22}) implies
\begin{equation*}
\bar{\mu}(S_k) = [\mathbb{P}^{w_k}_{1} \bar{\mu}](D_L)<\bar{\mu}(D_L),\;\;k=1,\ldots,Q.
\end{equation*}
Now,
\[\sum_{k=1}^Q p_k \bar \mu(S_k)=\sum_{k=1}^Q p_k [\mathbb{P}_1^{w_k} \bar \mu](D_L)]<\bar \mu(D_L).\]
Using (\ref{claim_lemma}), $m(S_k) \ne m(D_L)$ and since $S_k\subset
D_L$, we have $m(S_k)<m(D_L)$ for $k=1,\ldots, Q$.  Hence,
\[P_1=\sum_{k=1}^Q p_k \frac{m(S_k)}{m(D_L)}<1\]
is transient.

We prove the result for the general case, where ${\cal X}_1$ is a
finite partition, by contradiction.  Suppose $P_1$ is not transient. Then, using the
general result from the theory of finite Markov chains
\cite{Markovchains}, there exists at least one non-negative
invariant probability vector, $\nu$, such that
\begin{equation}
\nu\cdot P_1 = \nu.\label{eq:nup1nu}
\end{equation}
Let,
\begin{equation*}
S=\{x\in D_i:\;\nu_i>0\},\;\;\;\;S_\ell=\{x\in S: T_{w_\ell}(x)\in
S\},\label{eq:sands2}
\end{equation*}
for $\ell=1,\ldots, Q$.
It is claimed
\begin{equation}
 m(S_\ell)=m(S),\;\;\ell=1,\ldots, Q.\label{eq:claimbmbm}
\end{equation}
We first assume the claim to be true and show the desired
contradiction.  Clearly, $S_\ell \subset S$ and if the claim were
true, (\ref{claim_lemma}) shows
\begin{equation}
\bar\mu(S_\ell)=\bar\mu(S),\;\;\ell=1,\ldots, Q.\label{eq:bmbm1}
\end{equation}
Next, because $S\subset X_1$,
\begin{eqnarray*}
\mathbb{P}_1\bar{\mu}(S) &= \sum_{\ell=1}^Q p_\ell \bar{\mu}(T_{w_\ell}^{-1}(S)\cap X_1)\nonumber\\ & \ge
\sum_{\ell=1}^Q p_\ell \bar{\mu}(T_{w_\ell}^{-1}(S)\cap S)=\sum_{\ell=1}^Q p_\ell \bar \mu(S_\ell).
\end{eqnarray*}
This, together with Eq.~(\ref{eq:bmbm1}), gives
\begin{equation*}
\mathbb{P}_1\bar{\mu}(S) \ge \bar{\mu}(S)
\end{equation*}
for a set $S$ with positive Lebesgue measure.  This contradicts
Eq.~(\ref{eq:eqnll22}) and proves the theorem.

It remains to show the claim. Let $\{i_k\}_{k=1}^l$ be the indices
with $\nu_{i_k}>0$. Equation .~(\ref{eq:nup1nu}) gives
\begin{equation*}
\sum_{k=1}^l \nu_{i_k}\sum_{\ell=1}^Q p_\ell [P^{w_\ell}_1]_{i_k j_r} =
\nu_{j_r}\;\;\text{for}\;\;r=1,\hdots,l.
\end{equation*}
Taking a summation, $\sum_{r=1}^l$, on either side gives
\begin{equation*}
\sum_{k=1}^l \nu_{i_k} \sum_{r=1}^l \sum_{\ell=1}^Q p_\ell[P^{w_\ell}_1]_{i_k j_r} = 1.
\end{equation*}
Since individual entries are non-negative and $\nu$ is a
probability vector, this implies
\begin{equation*}
\sum_{r=1}^l\sum_{\ell=1}^Q p_\ell [P^{w_\ell}_1]_{i_k j_r} = 1 \;\;k=1,\hdots,l.
\end{equation*}
Using formula~(\ref{eq:pij}) for the
individual matrix entries, this gives
\begin{eqnarray*}
\sum_{r=1}^l\sum_{\ell=1}^Q p_\ell  m(T_{w_\ell}^{-1}(D_{j_r})\cap D_{i_k}) =
m(D_{i_k}),
\end{eqnarray*}
\begin{eqnarray*}
\sum_{\ell=1}^Q p_\ell \sum_{r=1}^l  m(T_{w_\ell}^{-1}(D_{j_r})\cap D_{i_k}) =
m(D_{i_k}).
\end{eqnarray*}
Since $D_{j_r}$ are disjoint for any fixed $w_\ell$,
\[\cup_{r}\{x: T_{w_\ell}(x)\in D_{j_r}\}=\{x: T_{w_\ell}(x)\in \cup_r D_{j_r}\},\]
i.e., $\cup T_{w_\ell}^{-1}(D_{j_r}) = T_{w_\ell}^{-1}(\cup D_{j_r})$. Furthermore, since $T_{w_\ell}$ is one-to-one, we have $(T^{-1}_{w_\ell}(D_{j_{r_1}})\cap D_{i_k})\cap (T^{-1}_{w_\ell}(D_{j_{r_2}})\cap D_{i_k})=\emptyset$ as

\[\{x\in D_{i_k}: T_{w_\ell}(x)\in D_{j_{r_1}}\}\cap \{x\in D_{i_k}: T_{w_\ell}(x)\in D_{j_{r_2}}\}=\emptyset\]
 for $r_1\neq r_2$. Hence, $\sum_{\ell=1}^Q p_\ell   m(T_{w_\ell}^{-1}(\cup_{r=1}^\ell D_{j_r})\cap D_{i_k}) =
m(D_{i_k})$. Therefore,
\[\sum_{\ell=1}^Q p_\ell   m(T_{w_\ell}^{-1}(\cup_{r=1}^\ell D_{j_r})\cap D_{i_k}) =
m(D_{i_k}). \]
 However, by
construction, $S=\cup_{r=1}^l D_{j_r}$ and, thus,
\[\sum_{\ell=1}^Q p_\ell   m(T_{w_\ell}^{-1}(S)\cap D_{i_k}) =
m(D_{i_k}). \]

Taking  a summation on both sides $\sum_{k=1}^l$, we obtain
\begin{eqnarray}\sum_{\ell=1}^Q p_\ell m(S_\ell)=m(S)=\sum_{\ell=1}^Q p_\ell m(S)\label{cc}.
\end{eqnarray}
Since $S_\ell \subset S$, we have $m(S_\ell)\leq m(S)$. From (\ref{cc}), we obtain
\[\sum_{\ell=1}^Q p_\ell (m(S_\ell)-m(S))=0.\]
Since $p_{\ell}>0$ and $m(S_\ell)\leq m(S)$, we conclude $m(S_\ell)=m(S)$.

2. Suppose $P_1$ is transient. To show the equilibrium point is coarse stable, we
proceed by contradiction.  Using
definition~\ref{def_coarsestability}, if the equilibrium point is not coarse stable, then
there exists an attractor set $B\subset U \subset X_1$ with a
sub-partition ${\cal S}=\{D_{s_1},...,D_{s_l}\}$, $S=\cup_{k=1}^l
D_{s_k}$ such that $B\subset S\subset U$ and $Prob\{T(x,\xi)\in S\}=1$ for all $x\in S$.
\[Prob\{ T(x,\xi)\in S\}=1=\sum_{\ell=1}^Q p_k \chi_{S}(T_{w_\ell}(x))\]
for $x\in S$. This implies  $S$ is left invariant by each  $T_{w_\ell}$ for $\ell=1,\ldots, Q$. Hence,
\begin{equation*}
P^{w_\ell}_{s_k j}=\frac{m(T_{w_\ell}^{-1}(D_j)\cap D_{s_k})}{m(D_{s_k})}=0,
\label{non-transience}
\end{equation*}
whenever $D_j\notin {\cal S}$. Since, $T_{w_\ell}:S\rightarrow S$,
\begin{equation*}
\sum_{\ell=1}^Q p_{\ell}\sum_{j=1}^{l}[P^{w_\ell}_1]_{s_i s_j}=1 \;\;\;\;\;\;i=1,...,l,
\end{equation*}
i.e., $P_1$ is a Markov matrix with respect to the finite partition,
${\cal S}$.  From the general theory of the Markov matrix
\cite{Markovchains}, there  exists an invariant probability
vector, $\nu$, such that $
\nu\cdot P_1^n = \nu$
for all $n>0$ and $P_1$ is not transient.
\end{proof}

There are various different equivalent ways of computing the finite dimensional approximation of  Lyapunov measure using the finite dimensional approximation of the P-F operator. In particular,  if $P_1$ is transient, then the Lyapunov measure $\bar \mu$ can be  computed using the following formulas.
\begin{enumerate}
\item Compute the Lyapunov measure using an infinite series formula,
\[\bar \mu=m\cdot (I-P_1)^{-1}=m+m\cdot P_1^2+m\cdot P_1^2+\ldots,\]
for row vector $m>0$ (element wise). For example $m$ could be taken as row vector of all ones corresponding to the finite dimensional approximation of Lebesgue measure.
\item Compute the Lyapunov measure as a solution of the linear program
\[\bar \mu (\alpha I-P_1)>0,\;\;\alpha<1.\]
\end{enumerate}
\section{Examples and Simulation}\label{section_examples}

%
%
%
%
%
\begin{example}
We consider an inverted pendulum example with a stochastic damping parameter, $\xi$.
\begin{eqnarray}\label{pend}
  \dot{x_{1}} &=& x_{2}. \nonumber \\
  \dot{x_{2}} &=& -\sin x_1-(\xi+0.7) x_{2}.
\end{eqnarray}
For the purpose of simulation, the continuous system is discretized with the time step of discretization, $\Delta t=0.1$. The phase space, $X$, for simulation is taken to be equal to $X=[-\pi,\pi]\times [-\pi,\pi]$. For the purpose of compacting the state space,  we identify $-\pi$ and $\pi$ along the $x$ and $y$ axis. The damping parameter,, $\xi$ is assumed  stochastic with zero mean and uniform distribution;  hence, of the form $[-\alpha,\alpha]$. In the following, we present simulation results by changing the variance, i.e., $\alpha$, of the random variable,  $\xi$, while keeping the mean zero and verifying the stochastic stability of the origin.

In Figs.  \ref{pend1} and \ref{pend2}, we show the plot for the Lyapunov measure with  random variable, $\xi$, supported on interval $[-0.5,0.5]$ and $[-0.75,0.75]$, respectively. The existence of Lyapunov measure implies  the origin is almost everywhere almost sure stable. When support of the random parameter, $\xi$, is increased to $[-1,1]$, the system fails to have Lyapunov measure and is  confirmed from the invariant measure plot shown in Fig. \ref{pend3}. From Fig. \ref{pend3}, we notice  the invariant measure has support around the origin.  This is in contrast to the invariant measure plots when $\xi$ is supported on $[-0.5,0.5]$ and $[-0.75, 0.75]$ and shown in Figs. \ref{pend1} and \ref{pend2}. The invariant measure in these cases is supported at the origin and is denoted by a blue dot at the origin in Figs. \ref{pend1} and \ref{pend2}.


\begin{figure}[h!]
    \center
\includegraphics[width=0.42\textwidth]{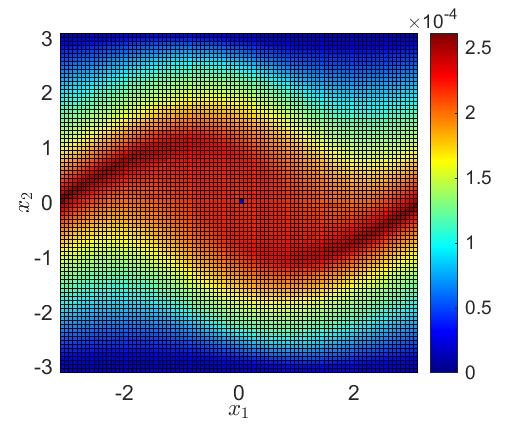}
    \vspace{-0.1in}
    \caption{  Lyapunov measure  plot  with support of $\xi\in[-0.5,0.5]$.}
    \label{pend1}
\end{figure}

\begin{figure}[h!]
    \center
    \includegraphics[width=0.42\textwidth]{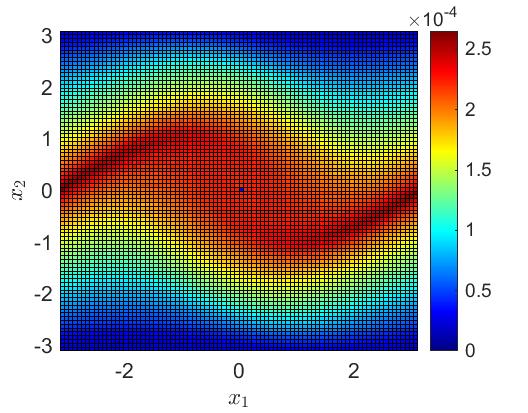}
    \vspace{-0.1in}
    \caption{  Lyapunov measure  plot  with support of $\xi\in[-0.75,0.75]$. }
    \label{pend2}
\end{figure}

%

\begin{figure}[h!]
    \center
    \includegraphics[width=0.42\textwidth]{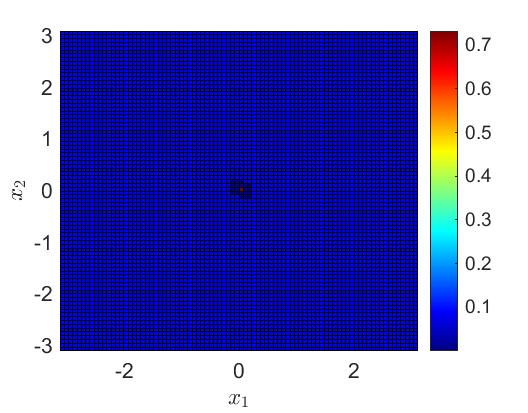}
    \vspace{-0.1in}
    \caption{ Invariant measure  plot  with support of $\xi\in[-1,1]$.}
     \label{pend3}
\end{figure}
\end{example}

\begin{example}
The second example is the stochastic counterpart of the deterministic almost everywhere stable system example from \cite{Rantzer01}.
\begin{eqnarray}\label{SE1}
  \dot{x} &=& -2x+x^2-y^2, \nonumber\\
  \dot{y} &=& -6y(1+\xi)+2xy,
\end{eqnarray}
where $\xi$ is a random variable with zero mean and uniformly distributed between $[-\alpha,\alpha]$.  The deterministic system has four equilibrium points $(0,0), (2,0) \;\text{and}\; (3, \pm \sqrt 3)$ of which the origin is stable equilibrium and the remaining are unstable. For the stochastic system, two of the equilibrium points - one at the origin and equilibrium point at $(2,0)$ are preserved. For the purpose of simulation, the system is discretized in time with the time step for discretization chosen to be equal to $\Delta t=0.1$. The state space $X$ is taken to be $X=[-4,4]\times [-4,4]$. For the purpose of compactness, $-4$ and $4$ are identified along the $x-$axis.

For better visualization, we plot the logarithm of the Lyapunov measure plot for varying variance of the random variable, $\xi$, by changing $\alpha$, while retaining its zero mean. In Figs. \ref{ran1} and  \ref{ran2}, we show the log plot for the Lyapunov measure with $\xi$ supported between $[-0.25,0.25]$ and $[-0.5,0.5]$, respectively. The existence of a Lyapunov measure in these plots implies  the origin is almost everywhere almost sure stable for these statistics of random variable.  When the support of the random variable is changed to $[-0.1,0.1]$, the origin is no longer almost everywhere almost sure stable. In Fig. \ref{ran3}, we show the plot for the invariant measure  supported in the neighborhood of the origin. This is in contrast to the invariant measure plot for the cases when $\xi\in [0.25,0.25]$ and $\xi\in[0.5,0.5]$. For both  cases, the invariant measure is supported only at the single cell containing the origin.

\begin{figure}[h!]
    \center
    \includegraphics[width=0.42\textwidth]{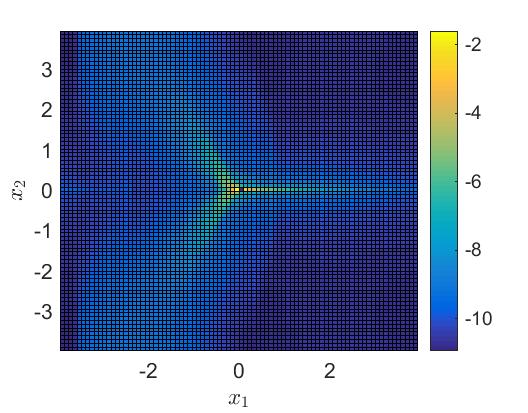}
    \vspace{-0.1in}
    \caption{ Log plot of Lyapunov measure with noise support $\xi\in [-0.25, 0.25]$.}
    \label{ran1}
\end{figure}

\begin{figure}[h!]
    \center
    \includegraphics[width=0.42\textwidth]{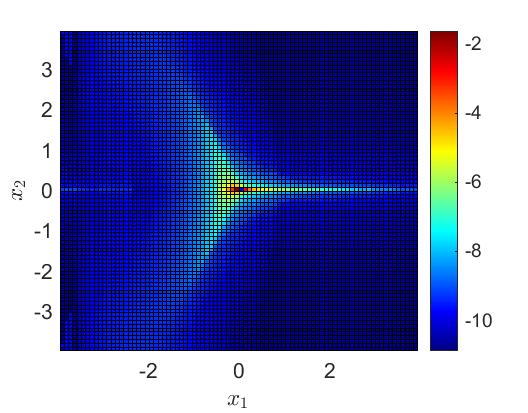}
    \vspace{-0.1in}
    \caption{ Log plot of Lyapunov measure with noise support $\xi\in [-0.5, 0.5]$.}
    \label{ran2}
\end{figure}


\begin{figure}[h!]
    \center
    \includegraphics[width=0.42\textwidth]{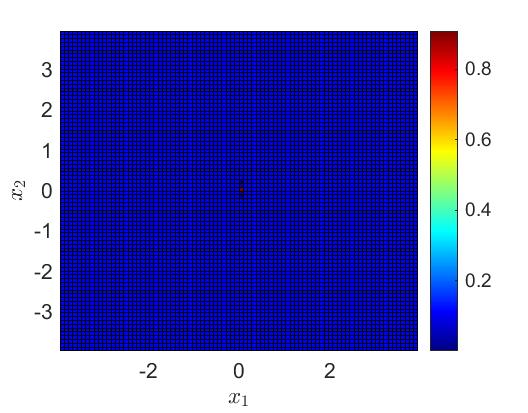}
    \vspace{-0.0in}
    \caption{ Invariant measure plot for $\xi\in [-1,1]$.}
     \label{ran3}
\end{figure}

\end{example}

\section{Conclusions}\label{section_conclusion}
Weaker set-theoretic notion of almost everywhere stability for stochastic dynamical system is introduced. Linear P-F operator-based Lyapunov measure as a new tool for stochastic stability verification is introduced. Duality between the Lyapunov function and the Lyapunov measure is established in a stochastic setting. Set-oriented numerical methods are proposed for the finite dimensional approximation of the Lyapunov measure. The finite dimensional approximation introduce further weaker notion of stability refereed to as coarse stochastic stability. The framework developed in this paper can be easily extended to the case where the noise, $\xi_n$, forms a Markov processes. This can be done by employing the P-F operator defined for systems with Markov noise processes \cite{froyland_extracting}.

\section{Acknowledgments}
The author would like to acknowledge Venketash Chinde from Iowa State University for help with simulations.

\bibliographystyle{IEEEtran}
\bibliography{ref,ref1,Umesh_ref}

\begin{thebibliography}{10}
\providecommand{\url}[1]{#1}
\csname url@samestyle\endcsname
\providecommand{\newblock}{\relax}
\providecommand{\bibinfo}[2]{#2}
\providecommand{\BIBentrySTDinterwordspacing}{\spaceskip=0pt\relax}
\providecommand{\BIBentryALTinterwordstretchfactor}{4}
\providecommand{\BIBentryALTinterwordspacing}{\spaceskip=\fontdimen2\font plus
\BIBentryALTinterwordstretchfactor\fontdimen3\font minus
  \fontdimen4\font\relax}
\providecommand{\BIBforeignlanguage}[2]{{%
\expandafter\ifx\csname l@#1\endcsname\relax
\typeout{** WARNING: IEEEtran.bst: No hyphenation pattern has been}%
\typeout{** loaded for the language `#1'. Using the pattern for}%
\typeout{** the default language instead.}%
\else
\language=\csname l@#1\endcsname
\fi
#2}}
\providecommand{\BIBdecl}{\relax}
\BIBdecl

\bibitem{Hasminskii_book}
R.~Z. Has'minski\u{i}, \emph{Stochastic Stability of differential
  equations}.\hskip 1em plus 0.5em minus 0.4em\relax {Germantown ,MD}: Sijthoff
  \& Noordhoff, 1980.

\bibitem{arnold_book_rds}
L.~Arnold, \emph{Random Dynamical Systems}.\hskip 1em plus 0.5em minus
  0.4em\relax {Berlin, Heidenberg}: {Springer Verlag}, 1998.

\bibitem{Kushner_book}
H.~J. Kushner, \emph{Stocahstic Stability and Control}.\hskip 1em plus 0.5em
  minus 0.4em\relax {New York}: Academic Press, 1967.

\bibitem{mao_book}
X.~Mao, \emph{Exponential Stability of Stochastic Differential
  Equations}.\hskip 1em plus 0.5em minus 0.4em\relax Monographs and Textbooks
  in Pure and Applied Mathematics Series, Marcel Dekker, Inc., 1994.

\bibitem{network_basar_tamar}
O.~Imer, S.~Yuksel, and T.~Basar, ``{Optimal control of LTI systems over
  communication networks},'' \emph{Automatica}, vol.~42, no.~9, pp. 1429--1440,
  2006.

\bibitem{networksystems_foundation_sastry}
{L. Schenato and B. Sinopoli and M. Franceschitti and K. Poolla and S. Sastry},
  ``{Foundations of control and estimation over Lossy networks},''
  \emph{{Proceedings of IEEE}}, vol.~95, no.~1, pp. 163--187, 2007.

\bibitem{scl04}
N.Elia, ``Remote stabilization over fading channels,'' \emph{Systems and
  Control Letters}, vol.~54, pp. 237--249, 2005.

\bibitem{amit_erasure_observation_journal}
A.~Diwadkar and U.~Vaidya, ``Limitation on nonlinear observation over erasure
  channel,'' \emph{{IEEE Transactions on Automatic Control}}, vol.~58, no.~2,
  pp. 454--459, 2013.

\bibitem{amit_ltv_journal}
{A. Diwadkar and U. Vaidya}, ``Stabilization of linear time varying systems
  over uncertain channels,'' \emph{{International Journal of Robust and
  Nonlinear Control}}, vol.~24, no.~7, pp. 1205--1220, 2014.

\bibitem{lure_stab_journal}
\BIBentryALTinterwordspacing
A.~Diwadkar, S.~Dasgupta, and U.~Vaidya, ``Control of systems in {L}ure form
  over erasure channels,'' \emph{Accepted for publication in International
  Journal of Robust and Nonlinear Control}, 2014. [Online]. Available:
  \url{http://dx.doi.org/10.1002/rnc.3231}
\BIBentrySTDinterwordspacing

\bibitem{sai_cdc_stochastic_linear}
S.~Pushpak, A.~Diwadkar, and U.~Vaidya, ``Stochastic stability analysis and
  controller systhesis for continuous time linear systems,'' in
  \emph{{Proceedings of IEEE Control and Decision Conference}}, Osaka, Japan,
  2015.

\bibitem{stochastic_kristic2}
H.~Deng, M.~Kristi\'{c}, and R.~J. Williams, ``Stabilization of stochastic
  nonlinear systems driven by noise of unknown covariance,'' \emph{IEEE
  Transactions of Automatic Control}, vol.~46, pp. 1237--1253, 2001.

\bibitem{stochastic_kristic3}
H.~Deng and M.~Kristi\'{c}, ``Output-feedback stabilization of stochastic
  nonlinear systems driven by noise of unknown covariance,'' \emph{Systems and
  Control Letters}, vol.~39, pp. 173--182, 2000.

\bibitem{Vaidya_erasure_SCL}
U.~Vaidya and N.~Elia, ``Limitation on nonlinear stabilization over packet-drop
  channels: Scalar case,'' \emph{{Systems and Control Letters}}, vol.~61,
  no.~9, pp. 959--966, 2012.

\bibitem{stochastic_florchingeer1}
P.~Florchinger, ``Lyapunov like techniques for stochastic stability,''
  \emph{SIAM J. Control Optim.}, vol.~33, pp. 1151--1169, 1995.

\bibitem{stochastic_florchingeer2}
------, ``A universal formula for the stabilization of control stochastic
  differential equation,'' \emph{Stochastic Annal. Appl.}, vol.~11, pp.
  155--162, 1993.

\bibitem{stochastic_florchingeer3}
------, ``Feedback stabilization of affine in the control stochastic
  differential systems by the control lyapunov function method,'' \emph{SIAM J.
  Control Optim.}, vol.~35, p. 500–511, 1997.

\bibitem{stochastic_kristic1}
H.~Deng and M.~Kristi\'{c}, ``Stocahstic nonlinear stabilization- part i:
  backstepping design,'' \emph{Systems and control letters}, vol.~32, pp.
  143--150, 1997.

\bibitem{stochastic_kristic4}
------, ``Stocahstic nonlinear stabilization- part {II}: Inverse optimality,''
  \emph{Systems and control letters}, vol.~32, pp. 151--159, 1997.

\bibitem{stochastic_kristic5}
------, ``Output-feedback stochastic nonlinear stabilization,'' \emph{IEEE
  Trans. Automat. Contr}, vol.~44, p. 328–333, 1999.

\bibitem{SOS_book}
D.~Henrion and A.~Garulli, Eds., \emph{Positive polynomials in control}, ser.
  Lecture Notes in Control and Information Sciences.\hskip 1em plus 0.5em minus
  0.4em\relax Berlin: Springer-Verlag, 2005, vol. 312.

\bibitem{Prajna04}
S.~Prajna, P.~A. Parrilo, and A.~Rantzer, ``Nonlinear control synthesis by
  convex optimization,'' \emph{{IEEE} Transactions on Automatic Control},
  vol.~49, no.~2, pp. 1--5, 2004.

\bibitem{Parrilothesis}
P.~A. Parrilo, ``Structured semidefinite programs and semialgebraic geometry
  methods in robustness and optimization,'' Ph.D. dissertation, California
  Institute of Technology, Pasadena, CA, 2000.

\bibitem{Junge_Osinga}
O.~Junge and H.~Osinga, ``A set oriented approach to global optimal control,''
  \emph{ESAIM: Control, Optimisation and Calculus of Variations}, vol.~10,
  no.~2, pp. 259--270, 2004.

\bibitem{Junge_scl_05}
L.~Grune and O.~Junge, ``A set oriented approach to optimal feedback
  stabilization,'' \emph{Systems Control Lett.}, vol.~54, no.~2, pp. 169–--180,
  2005.

\bibitem{cell-cell1}
L.~G. Crespo and J.~Q. Sun, ``Solution of fixed final state optimal control
  problem via simple cell mapping,'' \emph{Nonlinear dynamics}, vol.~23, pp.
  391--403, 2000.

\bibitem{hernandez}
D.~Hern\'{a}ndez-Hern\'{a}ndez, O.~Hern\'{a}ndez-Lerma, and M.~Taksar, ``The
  linear programming approach to deterministic optimal control problem,''
  \emph{Applicationes Mathematicae}, vol.~24, no.~1, pp. 17--33, 1996.

\bibitem{mey99a}
S.~P. Meyn, ``Algorithms for optimization and stabilization of controlled
  {M}arkov chains,'' \emph{S\=adhan\=a}, vol.~24, no. 4-5, pp. 339--367, 1999.

\bibitem{Lasserre_ocp}
J.~Lasserre, C.~Prieur, and D.~Henrion, ``Nonlinear optimal control: Numerical
  approximation via moment and {LMI}-relaxations,'' in \emph{Proceeding of IEEE
  {C}onference on {D}ecision and {C}ontrol}, Seville, Spain, 2005.

\bibitem{book_mceneaney}
W.~McEneaney, \emph{Max-plus methods for nonlinear control and
  estimation}.\hskip 1em plus 0.5em minus 0.4em\relax {Boston}: Birkhauser,
  2006.

\bibitem{Grune_04}
L.~Grune, ``Error estimation and adaptive discretization for the discrete
  stochastic {Hamilton-Jacobi-Bellman} equation,'' \emph{Numerische
  Mathematik}, vol.~99, pp. 85--112, 2004.

\bibitem{cell-cell2}
L.~G. Crespo and J.~Q. Sun, ``Stochastic optimal control via bellman
  principle,'' \emph{Automatica}, vol.~39, pp. 2109--2114, 2003.

\bibitem{Rantzer01}
A.~Rantzer, ``A dual to {L}yapunov's stability theorem,'' \emph{Systems \&
  Control Letters}, vol.~42, pp. 161--168, 2001.

\bibitem{almosteverywhere_stochastic}
R.~Van~Handel, ``{Almost global stochastic stability},'' \emph{{SIAM Journal on
  Control and Optimization}}, vol.~45, pp. 1297--1313, 2006.

\bibitem{angeli_aeinputputput}
D.~Angeli, ``An almost global notion of input-to-state stability,'' \emph{IEEE
  Transactions on Automatic Control}, vol.~49, pp. 866--874, 2004.

\bibitem{Vaidya_TAC}
U.~Vaidya and P.~G. Mehta, ``Lyapunov measure for almost everywhere
  stability,'' \emph{IEEE Transactions on Automatic Control}, vol.~53, no.~1,
  pp. 307--323, 2008.

\bibitem{Vaidya_CLM_journal}
U.~Vaidya, P.~Mehta, and U.~Shanbhag, ``Nonlinear stabilization via control
  lyapunov meausre,'' \emph{IEEE Transactions on Automatic Control}, vol.~55,
  no.~6, pp. 1314--1328, 2010.

\bibitem{arvind_ocp_journal_IEEE}
A.~Raghunathan and U.~Vaidya, ``Optimal stabilization using {L}yapunov
  measures,'' \emph{IEEE Transactions on Automatic Control}, vol.~59, no.~5,
  pp. 1316--1321, 2014.

\bibitem{mezic_koopman_stability}
A.~Mauroy and I.~Mezic´, ``A spectral operator-theoretic framework for global
  stability,'' in \emph{Proc. of IEEE Conference of Decision and Control},
  Florence, Italy, 2013.

\bibitem{Lasota}
A.~Lasota and M.~C. Mackey, \emph{Chaos, Fractals, and Noise: Stochastic
  Aspects of Dynamics}.\hskip 1em plus 0.5em minus 0.4em\relax New York:
  Springer-Verlag, 1994.

\bibitem{Dellnitztransport}
M.~Dellnitz, O.~Junge, W.~S. Koon, F.~Lekien, M.~Lo, J.~E. Marsden, K.~Padberg,
  R.~Preis, S.~D. Ross, and B.~Thiere, ``Transport in dynamical astronomy and
  multibody problems,'' \emph{International Journal of Bifurcation and Chaos},
  vol.~15, pp. 699--727, 2005.

\bibitem{Dellnitiz_almostinvariant}
G.~Froyland and M.~Dellnitz, ``Detecting and locating near-optimal
  almost-invariant sets and cycles,'' \emph{SIAM Journal on Scientific
  Computing}, vol.~24, no.~6, pp. 1839--1863, 2003.

\bibitem{mezic_koopmanism}
M.~Budisic, R.~Mohr, and I.~Mezic, ``Applied koopmanism,'' \emph{Chaos},
  vol.~22, pp. 047\,510--32, 2012.

\bibitem{Mezic_comparison}
I.~Mezi\'{c} and A.~Banaszuk, ``Comparison of systems with complex behavior,''
  \emph{Physica D}, vol. 197, pp. 101--133, 2004.

\bibitem{mezic_chaos}
I.~Mezi\'{c} and S.~Wiggins, ``A method for visualization of invariant sets of
  dynamical systems based on ergodic partition,'' \emph{Chaos}, vol.~9, no.~1,
  pp. 213--218, 1999.

\bibitem{froyland_ocean}
G.~Froyland, K.~Padberg, M.~England, and A.~M. Treguier, ``Detection of
  coherent oceanic structures using transfer operators,'' \emph{Physical Review
  Letters}, vol.~98, pp. 2\,245\,031--4, 2007.

\bibitem{froyland_extracting}
G.~Froyland, ``Extracting dynamical behaviour via {Markov} models,'' in
  \emph{Nonlinear Dynamics and Statistics: Proceedings, Newton Institute,
  Cambridge, 1998}, A.~Mees, Ed.\hskip 1em plus 0.5em minus 0.4em\relax
  Birkhauser, 2001, pp. 283--324.

\bibitem{vaidya_stochastic_lyap}
U.~Vaidya, ``Stochastic stability analysis of discrete-time system using
  lyapunov measure,'' in \emph{{Proceedings of American Control Conference}},
  Chicago, IL, 2015, pp. 4646--4651.

\bibitem{vaidya_stochastic_lyap_comp}
U.~Vaidya and V.~Chinde, ``Computation of lyapunov measure for stochastic
  stability verification,'' in \emph{{Accepted for publication in IEEE Control
  and Decision Conference}}, Osaka, Japan, 2015.

\bibitem{sean_meyn}
S.~P. Meyn and R.~L. Tweedie, \emph{Markov Chains and Stochastic
  Stability}.\hskip 1em plus 0.5em minus 0.4em\relax London: Springer-Verlag,
  1993.

\bibitem{Dellnitz_Junge}
M.~Dellnitz and O.~Junge, ``On the approximation of complicated dynamical
  behavior,'' \emph{SIAM Journal on Numerical Analysis}, vol.~36, pp. 491--515,
  1999.

\bibitem{Dellnitz00}
------, \emph{Set oriented numerical methods for dynamical systems}.\hskip 1em
  plus 0.5em minus 0.4em\relax World Scientific, 2000, pp. 221--264.

\bibitem{Markovchains}
J.~Norris, \emph{Markov Chains}, ser. Cambridge Series in Statistical and
  Probabilistic Mathematics.\hskip 1em plus 0.5em minus 0.4em\relax Cambridge:
  Cambridge University Press, 1997.

\end{thebibliography}
\end{document}